\documentclass[12pt]{amsart}
\usepackage[utf8]{inputenc}

\usepackage{amsmath,amssymb,amsthm}

\usepackage[unicode]{hyperref}
\usepackage{color}

\newtheorem{thm}{Theorem}[section] 

\newtheorem{cor}[thm]{Corollary}

\newtheorem{lem}[thm]{Lemma}
\newtheorem{prop}[thm]{Proposition}

\theoremstyle{definition}
\newtheorem{rem}[thm]{Remark}
\newtheorem{exmpl}[thm]{Example}

\newcommand{\noopsort}[1]{}

\title{Recurrence relations over division algebras}
\author{Adam Chapman}
\email{adam1chapman@yahoo.com}
\address{School of Computer Science, Academic College of Tel-Aviv-Yaffo, Rabenu Yeruham St., P.O.B 8401 Yaffo, 6818211, Israel}
\author{Solomon Vishkautsan}
\email{wishcow@gmail.com}
\address{Department of Computer Science, Tel-Hai University of Kiryat Shmona in the Galilee, Upper Galilee, Tel-Hai 1220800, Israel}
 

\begin{document}

\begin{abstract}
We generalize the solution of linear recurrence relations from fields to central division algebras, adapting the standard tools of companion matrices and characteristic polynomials to the non-commutative setting. We then solve linear recurrences of order $2$ over octonion division algebras. 
\end{abstract}

\maketitle

\section{Introduction}

A linear recurrence defines a family of sequences for which each term is a fixed linear combination of preceding terms in the sequence. It is the simplest type of recurrence relation, where each term is a function of preceding terms. We can write a general linear recurrence of order $n$ (i.e., every element in the sequence is determined by the $n$ preceding terms) in the following way:
\begin{equation}\label{eq:linear-recurrence}
	a_{k+n} = -c_0a_k-c_1a_{k+1} - \cdots -c_{n-1}a_{k+n-1} ,
\end{equation}
where $k\ge 0$, $c_0\ne 0$ and the coefficients $c_0,\ldots,c_{n-1}$ belong to some field $F$ (the choice of negative signs is technical). It is clear, by induction, that given the first $n$ terms $a_0,a_1,\ldots,a_{n-1}$ of the sequence in $F$, which are called \emph{initial conditions} for the recurrence relation~\eqref{eq:linear-recurrence}, there exists a unique sequence, defined over $F$, satisfying the linear recurrence. The existence and uniqueness of a solution to \eqref{eq:linear-recurrence} are still true when the field $F$ is replaced by any ring $R$. 

A classical example of a linear recurrence is the Fibonacci sequence,
\[a_{n+2} = a_n+a_{n+1}, \quad a_0=0, a_1=1.\]
Its closed-form solution is Binet's formula
\[a_n = \frac{\varphi^n-\psi^n}{\sqrt{5}}, \quad \text{where } \varphi = \frac{1+\sqrt{5}}{2},\psi = \frac{1-\sqrt{5}}{2}.\]
This example shows that despite all elements in the sequence being rational, the closed-form expression may require us to extend the field $F$. This issue is redundant if the field $F$ is algebraically closed. 

Beyond their intrinsic algebraic interest, linear recurrence relations are ubiquitous in applied mathematics, with applications in cryptography, pseudorandom number generation and error correcting codes among others (e.g.\ \cite{niederreiter1992random, golomb1982shift, mullen2013handbook}). Recent developments also highlight the role of noncommutative algebra in applied areas, ranging from skew-polynomial based coding theory and quaternion neural networks to division-algebraic models in physics (e.g.\ \cite{gluesing2021introduction, parcollet2020,   furey2022division,furey2018three}). By extending linear recurrence relation solution methods to division algebras, we aim to provide a framework that may support both theoretical advances and future applied investigations. 

The determination of a closed-form for a linear recurrence with given initial conditions is a classical topic in mathematics, usually taught in discrete mathematics courses to first year students (e.g.\ \cite{rosen_discrete_2011, grimaldi_discrete_2006}). The approach most commonly taken is to define the \emph{characteristic polynomial} of the linear recurrence \eqref{eq:linear-recurrence} to be the polynomial
\begin{equation} \label{eq:characteristic-polynomial}
p(x) = x^n + c_{n-1}x^{n-1} + \cdots + c_1x + c_0.
\end{equation}
If the characteristic polynomial has $n$ distinct roots $\lambda_1,\ldots,\lambda_n$ in $\bar{F}$ (where $\bar{F}$ is the algebraic closure of $F$), then one can show that there exist unique $b_1,\ldots,b_n\in \bar{F}$ such that the solution to the linear recurrence with initial conditions $a_0,\ldots,a_{n-1}$ is given by the sequence 
\begin{equation}\label{eq:superposition-distinct} a_k = b_1\lambda_1^k+\cdots+b_n\lambda_n^k\end{equation}

In the more complicated case when there exist multiple roots of the characteristic polynomial $p(x)$, the solution to the linear recurrence assumes the form 
\begin{equation}\label{eq:solution-with-polynomials}
	a_k = p_1(k)\lambda_1^k+\cdots+p_t(k)\lambda_t^k,
\end{equation}
where $\lambda_1,\ldots,\lambda_t$ are the distinct roots of $p(x)$ with respective multiplicities $m_1,\ldots,m_t$, and $p_1(x),\ldots,p_t(x)$ are polynomials of degrees less than $m_1,\ldots,m_t$, respectively. 

A more elegant approach to solving the linear recurrence is to consider the companion matrix of the linear recurrence~\eqref{eq:linear-recurrence}, defined as
\begin{equation}\label{eq:companion-matrix}
	A = \begin{bmatrix}
		0 & 1 & 0 & \cdots & 0 \\
		0 & 0 & 1 & \cdots & 0 \\
		\vdots & \vdots & \vdots & \ddots & \vdots \\
		0 & 0 & 0 & \cdots & 1 \\
		-c_0 & -c_1 & -c_2 & \cdots & -c_{n-1}
	\end{bmatrix}.
\end{equation}
The $n\times{n}$ matrix $A$ is the transpose of the companion matrix as usually defined in linear algebra. Moreover, $p(x)$ is in fact the characteristic polynomial of the matrix $A$; thus, we have a correspondence between the eigenvalues of $A$ and the roots of $p(x)$. The companion matrix satisfies the matrix relation
\begin{equation} \label{eq:matrix-relation}
	\begin{bmatrix}
	a_{k+1} \\ a_{k+2} \\ \vdots \\ a_{k+n-1} \\ a_{k+n}
	\end{bmatrix}
	= 
	A
	\begin{bmatrix}
	a_{k} \\ a_{k+1} \\ \vdots \\ a_{k+n-2} \\ a_{k+n-1}
	\end{bmatrix}.
\end{equation}
Therefore a solution to the linear recurrence can be obtained by computing the general $k$th power of the companion matrix $A$, as
\begin{equation} \label{eq:kth-matrix-relation}
	\begin{bmatrix}
	a_{k} \\ a_{k+1} \\ \vdots \\ a_{k+n-2} \\ a_{k+n-1}
	\end{bmatrix}
	= 
	A^k
	\begin{bmatrix}
	a_{0} \\ a_{1} \\ \vdots \\ a_{n-2} \\ a_{n-1}
	\end{bmatrix}.
\end{equation}
This can of course be computed by diagonalizing $A$ if possible, or by bringing $A$ to its Jordan canonical form (see Section~\ref{sec:prelim} for details). 

When $R$ is non-commutative, the recurrence relation defined in \eqref{eq:linear-recurrence} should be called a \emph{left} linear recurrence, which differs, for example, from the right linear recurrence
\begin{equation} \label{right-linear}
	a_{k+n} = -a_kc_0-a_{k+1}c_1 - \cdots -a_{k+n-1}c_{n-1}.
\end{equation}  

	We recall that a central division algebra is a division algebra $\mathcal{D}$ which is finite dimensional over its center $F$, a field. The dimension $[\mathcal{D}:F]$ is a square integer, and we call $d=\sqrt{[\mathcal{D}:F]}$ the degree of the algebra $\mathcal{D}$ over $F$. Unless otherwise stated, the algebras in this article are assumed to be associative and unital. The notion of an eigenvalue generalizes to left and right eigenvalues for a matrix defined over a division algebra, and these play different roles in the theory (See Section~\ref{sec:division-prelim} for details and \cite{jacobson1996finite} for general reference on central division algebras). 
	
In this article we show that the solution method of using the companion matrix \eqref{eq:companion-matrix} generalizes to left linear recurrences over central division algebras under appropriate conditions, and that under these conditions, the classical methods for solving linear recurrences over fields follow through smoothly. One cannot expect these methods to work in general since, for example, the Jordan canonical form does not exist for general division algebras. 

Now we come to a subtle point in terminology. Over division algebras, the characteristic polynomial of a matrix has a different meaning, and it is not equal to the polynomial in \eqref{eq:characteristic-polynomial}; its precise definition will be recalled in Section~\ref{sec:division-prelim}. To avoid confusion, we refer to the polynomial $p(x) = x^n + c_{n-1}x^{n-1} + \cdots + c_1x + c_0$ associated with a left linear recurrence \eqref{eq:linear-recurrence} defined over a division algebra as its \emph{primitive characteristic polynomial}, to distinguish it from the characteristic polynomial of the companion matrix.

Our results are as follows:

\begin{thm} \label{thm:diagonal}
	Let $A$ be the companion matrix associated with a left linear recurrence \eqref{eq:linear-recurrence} of order $n$ with coefficients in a central division algebra $\mathcal{D}$, and with fixed initial conditions $a_0,\ldots,a_{n-1}\in \mathcal{D}$. Suppose that there exist $n$ distinct left (right) eigenvalues of $A$, denoted by $\lambda_1,\ldots,\lambda_n \in \mathcal{D}$, such that no three lie in the same conjugacy class. Let $V=V_n(\lambda_1,\ldots,\lambda_n)$ be the Vandermonde matrix (see Section~\ref{sec:prelim}). Then $V$ is invertible, and setting
	\begin{equation}
		\begin{bmatrix}
	b_1 \\ b_2 \\ \vdots \\ b_n
	\end{bmatrix} = V^{-1}	\begin{bmatrix}
	a_0 \\ a_1 \\ \vdots \\ a_{n-1}
	\end{bmatrix},
\end{equation}
the unique solution to the left linear recurrence is given by $a_k = \lambda_1^kb_1+\cdots+\lambda_n^kb_n$ (note that the coefficients appear on the right).
\end{thm}

\begin{thm} \label{thm:jordan}
	Let $A$ be the companion matrix associated with a left linear recurrence \eqref{eq:linear-recurrence} of order $n$ with coefficients in a central division algebra $\mathcal{D}$ of characteristic $0$, with fixed initial conditions $a_0,\ldots,a_{n-1}\in \mathcal{D}$. Let $\phi(x)$ be the minimal polynomial of $A$ over $F$, the center of $\mathcal{D}$. If $\phi(x)$ factors in $\mathcal{D}[x]$ to a product of linear factors, then there exist distinct left eigenvalues $\lambda_1,\ldots,\lambda_t \in \mathcal{D}$ of $A$, positive integers $m_1,\ldots,m_t$ with $\sum_{i=1}^tm_i = n$ and an invertible matrix $U$ such that if we set 	
		\begin{equation}
			\begin{bmatrix}
				b_1 \\ b_2 \\ \vdots \\ b_n
			\end{bmatrix} = U^{-1}	
			\begin{bmatrix}
				a_0 \\ a_1 \\ \vdots \\ a_{n-1}
			\end{bmatrix},
		\end{equation}
	we get that the unique solution to the linear recurrence is of the form 
	\begin{align*}
		a_k  =& p_1(k)\lambda_1^kb_1+\cdots+p_{m_1}(k)\lambda_1^kb_{m_1} \\
		     &+  p_{m_1+1}(k)\lambda_2^kb_{m_1+1}+\cdots+p_{m_2}(k)\lambda_2^kb_{m_2} \\
		     & \cdots \\
		     &+  p_{m_{t-1}+1}(k)\lambda_t^kb_{m_{t-1}+1}+\cdots+p_{n}(k)\lambda_t^kb_{n}, 
	\end{align*} 
where $p_1(x),\ldots,p_n(x)$ are polynomials in $\mathcal{D}[x]$ such that $\deg(p_j(x)) < m_i$ when $m_{i-1} < j \le m_i$. 
	\end{thm}
The characteristic $0$ condition on the algebra $\mathcal{D}$ in Theorem~\ref{thm:jordan} can be removed if we assume that the eigenvalues of $A$ are separable over $F$. 

Finally, in Section~\ref{sec:octonions} we take a walk on the wild side, and consider left linear recurrences over non-associative division algebras. Specifically, we prove that for left linear recurrences of order $2$ over octonion division algebras, the matrix relation~\eqref{eq:matrix-relation} decomposes into a matrix relation over a quaternion subalgebra, allowing us to explicitly solve the recurrence.

\begin{thm}\label{thm:octonion} 
	Let $p(x)$ be the primitive characteristic polynomial associated with a left linear recurrence of order $2$, with coefficients in an octonion division algebra $O$, along with fixed initial conditions $a_0,a_1\in O$. Assume that $p(x)$ has two distinct roots $\lambda_1,\mu_1$. 	
Let $\overline{p(x)}$ be the conjugate (octonion conjugation) polynomial of $p(x)$. Then $p(x)$  also has two distinct roots $\lambda_2,\mu_2$ (in the same conjugacy class as $\lambda_1,\mu_1$, respectively), and 
	there exist constants $b_1,b_2,c_1,c_2$ in a quaternion subalgebra $Q$ containing the coefficients of $p(x)$, such that the solution of the left linear recurrence is given by 
	\[a_k = \lambda_1^kb_1 + \mu_1^kb_2 + (c_1\bar{\lambda}_2^k+c_2\bar{\mu}_2^k)\ell,\] 
	where $\ell$ is some fixed element in $O$ (of trace $\ell + \bar{\ell}=0$) such that $O=Q\oplus Q\ell$. 
\end{thm}

\begin{thm}\label{thm:octonion2} 
Let $p(x)$ be the primitive characteristic polynomial associated with a left linear recurrence of order $2$, with coefficients in an octonion division algebra $O$ of characteristic $0$, along with fixed initial conditions $a_0,a_1\in O$.
	Assume that $p(x)$ has a unique root $\lambda \in O$, and let $\overline{p(x)}$ be its conjugate. Then $\overline{p(x)}$ also has a unique root $\mu$ (in the same conjugacy class as $\lambda$), 	
	and there exist constants $b_1,b_2,b_3,b_4$ and linear polynomials $p_1(x),p_2(x),p_3(x),p_4(x)$ over the quaternion subalgebra $Q$ generated by the coefficients of $p(x)$ such that the solution of the left linear recurrence is given by 
	\[a_k = p_1(k)\lambda^kb_1 + p_2(k)\lambda^kb_2 + (b_3\bar{\mu}^kp_3(k)+b_4\bar{\mu}^kp_4(k))\ell,\]
		where $\ell$ is some fixed element in $O$ (of trace $\ell + \bar{\ell}=0$) such that $O=Q\oplus Q\ell$.  
\end{thm}

\section{Linear recurrences over fields} \label{sec:prelim}


In this section we recall the classical method of solving linear recurrence relations over fields using the companion matrix. 
Recall that in the classical case of a linear recurrence relation over a field $F$, which we assume to be algebraically closed for simplicity, the recurrence can be described by the following matrix relation:

\begin{equation} \label{eq:matrix-relation2}
	\begin{bmatrix}
	a_{k+1} \\ a_{k+2} \\ \vdots \\ a_{k+n-1} \\ a_{k+n}
	\end{bmatrix}
	= 
	\begin{bmatrix}
		0 & 1 & 0 & \cdots & 0 \\
		0 & 0 & 1 & \cdots & 0 \\
		\vdots & \vdots & \vdots & \ddots & \vdots \\
		0 & 0 & 0 & \cdots & 1 \\
		-c_0 & -c_1 & -c_2 & \cdots & -c_{n-1}
	\end{bmatrix}
	\begin{bmatrix}
	a_{k} \\ a_{k+1} \\ \vdots \\ a_{k+n-2} \\ a_{k+n-1}
	\end{bmatrix}.
\end{equation}
The square $n$ by $n$ matrix $A$ in \eqref{eq:matrix-relation2} is the companion matrix of the linear recurrence \eqref{eq:linear-recurrence}; as mentioned in the introduction it is the transpose of the companion matrix to the characteristic polynomial $p(x) = c_0+c_1x+\cdots+c_{n-1}x^{n-1}+x^n$. 

Given an eigenvalue $\lambda$ of the matrix $A$, i.e., any root of the characteristic polynomial $p(x)$, we get a corresponding eigenvector $$v=	\begin{bmatrix}
	1 \\ \lambda \\ \vdots \\ \lambda^{n-2} \\ \lambda^{n-1}
	\end{bmatrix},$$
as multiplying by $A$, we get (as we know that $p(\lambda)=0$) 
\[Av = \begin{bmatrix}
	\lambda \\ \lambda^2 \\ \vdots \\ \lambda^{n-1} \\ -c_0-c_1\lambda-\ldots-c_{n-1}\lambda^{n-1}
	\end{bmatrix} = \begin{bmatrix}
	\lambda \\ \lambda^2 \\ \vdots \\ \lambda^{n-1} \\ \lambda^{n}
	\end{bmatrix} = \lambda{v}.\]
This implies that $a_k = \lambda^k$ is a solution to the general recurrence relation \eqref{eq:linear-recurrence} (i.e. without initial conditions). 

Given $t$ distinct eigenvalues $\lambda_1,\ldots,\lambda_t$, we can do a superposition of solutions, and obtain solutions to the general recurrence of the form 
\begin{equation} \label{eq:superposition}
	a_k = \alpha_1\lambda_1^k+\cdots+\alpha_t\lambda_t^k
\end{equation}
for any $\alpha_1,\ldots,\alpha_t\in F$.

If the matrix $A$ has $n$ distinct eigenvalues it is diagonalizable. This is equivalent to the characteristic polynomial having $n$ distinct roots.

\begin{prop} \label{prop:diagonal-field}
Suppose that the $n\times{n}$ companion matrix $C$ has $n$ distinct eigenvalues $\lambda_1,\ldots,\lambda_n$. Then we get a unique solution to the recurrence relation \eqref{eq:linear-recurrence}, given initial conditions $a_0,\ldots,a_{n-1} \in F$. 
\end{prop}

\begin{proof}
The condition of $n$ distinct eigenvalues is equivalent to any of the following  conditions:
\begin{enumerate}
	\item The characteristic polynomial $p(x)$ has $n$ distinct roots.
	\item The polynomial $p(x)$ is separable.
	\item The discriminant of $p(x)$ is non-zero. 
	\item The Vandermonde matrix of the roots $\lambda_1,\ldots,\lambda_n$ of $p(x)$
	\[ V = V_n(\lambda_1,\ldots,\lambda_n) = 
\begin{bmatrix}
1 & 1 & \cdots & 1 \\
\lambda_1 & \lambda_2 & \cdots & \lambda_n \\
\lambda_1^2 & \lambda_2^2 & \cdots & \lambda_n^2 \\
\vdots & \vdots & \ddots & \vdots \\
\lambda_1^{n-1} & \lambda_2^{n-1} & \cdots & \lambda_n^{n-1}
\end{bmatrix}
\] 	is invertible.
	\item The determinant of $V$ is non-zero.
\end{enumerate}

This implies that $C$ is diagonalizable and we have $C=VDV^{-1}$, where 
\[D=\begin{bmatrix}
\lambda_1 & 0   & \cdots & 0 \\
0   & \lambda_2 & \cdots & 0 \\
\vdots & \vdots & \ddots & \vdots \\
0   & 0   & \cdots & \lambda_n
\end{bmatrix}.\]

The unique solution is given by $a_k = \alpha_1\lambda_1^k+\ldots+\alpha_n\lambda_n^k$, where 
\begin{equation}
		\begin{bmatrix}
	\alpha_1 \\ \alpha_2 \\ \vdots \\ \alpha_n
	\end{bmatrix} = V^{-1}	\begin{bmatrix}
	a_0 \\ a_1 \\ \vdots \\ a_{n-1}
	\end{bmatrix}.
\end{equation}
\end{proof}

Thus, in the case of $n$ distinct eigenvalues, the recurrence relation can be computed efficiently by
\begin{equation}
	\begin{bmatrix}
	a_{k+1} \\ a_{k+2} \\ \vdots \\ a_{k+n}
	\end{bmatrix}
	= 
C^k	\begin{bmatrix}
	a_0 \\ a_1 \\ \vdots \\ a_{n-1}
	\end{bmatrix} = VD^kV^{-1} \begin{bmatrix}
	a_0 \\ a_1 \\ \vdots \\ a_{n-1}
	\end{bmatrix},
\end{equation}
using the notations from the proof. 

Now we deal with the general case, i.e., when $A$ does not have $n$ distinct eigenvalues. 
Suppose that $\lambda$ is a multiple root of $p(x)$ of multiplicity $m$. In this case, since the geometric multiplicity of $\lambda$ is always $1$ for a companion matrix (see \cite[Theorem~3.3.14]{Horn-Johnson2013}), 
the matrix $C$ is no longer diagonalizable, and there is a unique Jordan block corresponding to $\lambda$, of the form
\[J_\lambda = \begin{bmatrix}
\lambda & 1   & \cdots & 0 \\
0   & \lambda & \ddots & 0 \\
\vdots & \vdots & \ddots & 1 \\
0   & 0   & \cdots & \lambda
\end{bmatrix}_{m\times m}.\]

Taking the $k$th power of $J_\lambda$, we get
\begin{equation} \label{eq:kth-power-jordan-block}
J_\lambda^k = 
\begin{bmatrix}
\lambda^k & \binom{k}{1} \lambda^{k-1} & \binom{k}{2} \lambda^{k-2} & \cdots & \binom{k}{m-1} \lambda^{k - (m-1)} \\
0 & \lambda^k & \binom{k}{1} \lambda^{k-1} & \cdots & \binom{k}{m-2} \lambda^{k - (m-2)} \\
0 & 0 & \lambda^k & \cdots & \binom{k}{m-3} \lambda^{k - (m-3)} \\
\vdots & \vdots & \vdots & \ddots & \vdots \\
0 & 0 & 0 & \cdots & \lambda^k
\end{bmatrix}.
\end{equation}

Thus if $A$ has exactly $t$ distinct eigenvalues $\lambda_1,\ldots,\lambda_t$ with respective multiplicities $m_1,\ldots,m_t$, we denote the respective Jordan blocks by $J_{m_1}(\lambda_1),\ldots,J_{m_t}(\lambda_1)$. Then there exists an invertible matrix $U$ such that $A = UJU^{-1}$ where $J$ is the block diagonal matrix with blocks $J_{m_1}(\lambda_1),\ldots,J_{m_t}(\lambda_t)$. Now it is clear that we get 
\begin{equation}
	\begin{bmatrix}
	a_{k} \\ a_{k+1} \\ \vdots \\ a_{k+n-1}
	\end{bmatrix}
	= 
C^k	\begin{bmatrix}
	a_0 \\ a_1 \\ \vdots \\ a_{n-1}
	\end{bmatrix} = UJ^kU^{-1} \begin{bmatrix}
	a_0 \\ a_1 \\ \vdots \\ a_{n-1}
	\end{bmatrix}.
\end{equation}
Setting 
	\begin{equation}
		\begin{bmatrix}
	b_1 \\ b_2 \\ \vdots \\ b_n
	\end{bmatrix} = U^{-1}	\begin{bmatrix}
	a_0 \\ a_1 \\ \vdots \\ a_{n-1}
	\end{bmatrix},
\end{equation}
we get, using \eqref{eq:kth-power-jordan-block}, the solution \eqref{eq:solution-with-polynomials} from the introduction.

\section{Left and right eigenvalues in division algebras} \label{sec:division-prelim}

In this section, we collect known results from the literature that generalize the tools used in Section~\ref{sec:prelim} to division algebras. 

Let $\mathcal{D}$ be a (associative) division algebra over a field $F$. Then $\mathcal{D}^n$ has a natural structure of a $\mathcal{D}$-bimodule, which, by convention, we call a vector space over $\mathcal{D}$. As usual, we let $M_n(\mathcal{D})$ denote the $F$-algebra of $n$ by $n$ matrices with coefficients in $\mathcal{D}$. 
 A matrix in $M_n(\mathcal{D})$ is an endomorphism of $\mathcal{D}^n$ when viewed as a right module over $\mathcal{D}$. The ring $M_n(\mathcal{D})$ can be embedded in $M_{nd}(K)$, where $K$ is a maximal subfield of $\mathcal{D}$ and $d=[\mathcal{D}:K]$ (not to be confused with the degree of $\mathcal{D}$ as defined in the introduction). We denote this embedding by $f:M_n(\mathcal{D})\to M_n(K)$, and define the \emph{characteristic polynomial} of $A\in M_n(\mathcal{D})$ by $\Phi_A(x)=\det(f(A)-xI)$. 

We denote by $\mathcal{D}[x]$ the algebra of polynomials in a central indeterminate $x$, with all coefficients on the left.
For a polynomial $f(x)=c_nx^n+\cdots +c_1x+c_0\in \mathcal{D}[x]$ and an element $\lambda\in{\mathcal{D}}$ we define $f(\lambda)=c_n\lambda^n+\cdots +c_1\lambda+c_0.$ An element $\lambda\in{\mathcal{D}}$ is called a root of $f(x)\in \mathcal{D}[x]$ if $f(\lambda)=0$. It is well-known that $\lambda$ is a root of $f(x)$ if and only if there exists a polynomial $g(x)\in \mathcal{D}[x]$ such that $f(x)=g(x)(x-\lambda)$ (see \cite[Theorem 1]{gordon-motzkin1965}). 

The classical notion of eigenvalues and eigenvectors for matrices over fields extends to matrices over division algebras in two variants: given a division algebra $\mathcal{D}$ and a matrix $B \in M_n(\mathcal{D})$, an element $\lambda \in \mathcal{D}$ is a left\footnote{In many texts, e.g.\ \cite[Chapter 8]{cohn:1995}, ``our'' left eigenvalue is called a singular eigenvalue, while  left eigenvalue refers to a different type of eigenvalue.} (or right, respectively) eigenvalue of $B$ if there exists a nonzero $v \in \mathcal{D}^n$ for which $Bv=\lambda v$ (or $Bv=v\lambda)$.
When $\mathcal{D}$ is an associative division algebra, the right eigenvalues are relatively well-understood (see \cite{ChapmanMachen:2017}). 

The situation is more complicated for left eigenvalues, but nevertheless, a complete description was obtained for $2\times 2$ matrices over the real quaternion algebra $\mathbb{H}$ in \cite{HuangSo:2001}, and without any significant modification, their technique actually applies to any quaternion division algebra.

Two elements $d_1,d_2$ in a division algebra $\mathcal{D}$ are conjugate if there exists some nonzero $q\in \mathcal{D}$ such that $d_1 = qd_2q^{-1}$. The conjugacy class of $d_1$ in $\mathcal{D}$ is the set of all its conjugates. It is a well-known fact that any conjugate of a right eigenvalue of a matrix in $M_n(\mathcal{D})$ is a right eigenvalue (see e.g.\ \cite{cohn:1995}). 
	
We recall the following theorem regarding the eigenvalues of companion matrices. 

\begin{thm}[Chapman--Machen\cite{ChapmanMachen:2017}]\label{thm:cm}
Let $\mathcal{D}$ be a division algebra, and 
let $C$ be a companion matrix as in \eqref{eq:matrix-relation2} with coefficients in $\mathcal{D}$, and let $p(x)=x^n+c_{n-1}x^{n-1}+\cdots+c_0$ be its primitive characteristic polynomial~\footnote{In \cite{ChapmanMachen:2017} the polynomial $p(x)$ is not given a name.}. The left eigenvalues of $C$ are exactly the roots of $p(x)$, and for every root $\lambda$, a corresponding eigenvector is \[v=\begin{bmatrix}
	1 \\ \lambda \\ \lambda^2 \\ \vdots \\ \lambda^{n-1}
\end{bmatrix}.\] 
Moreover, every left eigenvalue is also a right eigenvalue, and every right eigenvalue of $C$ is in the conjugacy class of some left eigenvalue of $C$. 
\end{thm}

As we saw in \S\ref{sec:prelim} the Vandermonde matrix $V_n(a_1,\ldots,a_n)$ is  central to solving linear recurrence relations. If the values of $a_1,\ldots,a_n$ in the Vandermonde matrix are in some division ring $\mathcal{D}$, rather than lying in a field, then 
the values $a_1,\ldots,a_n$ being distinct no longer guarantees nonsingularity of the Vandermonde matrix, as illustrated by the following example over Hamilton's quaternions $\mathbb{H}$:
\[
V_3(i, j, k) = 
\begin{pmatrix}
1 & 1 & 1 \\
i & j & k \\
-1 & -1 & -1
\end{pmatrix}.
\]

Lam~\cite{lam1986} proved the following sufficient condition for the Vandermonde matrix to be nonsingular. 

\begin{thm}[Lam] \label{thm:lam}
	Let $a_1,\ldots,a_n$ be distinct elements in a division ring $\mathcal{D}$. If no three of $a_1,\ldots,a_n$ lie in one conjugacy class of $\mathcal{D}$, then the Vandermonde matrix $V_n(a_1,\ldots,a_n)$ is nonsingular. 
\end{thm} 

\begin{cor}
	For any two distinct elements $a,b$ in a division ring $\mathcal{D}$ the Vandermonde matrix $V_2(a,b)$ is invertible. 
\end{cor}

We remark (as is well-known) that as in linear algebra over a field, a matrix over a division algebra is nonsingular if and only if it is invertible (see e.g.\ \cite[Theorem 3.13.2]{ehrlich2013}). 

Diagonalization (where possible) of a matrix $A\in M_n(\mathcal{D})$ is performed using \emph{right} eigenvalues, as can easily be seen by the matrix relation $AU=UD$, where $D \in M_n(\mathcal{D})$ is a diagonal matrix and $U$ is a transformation matrix. Combining Theorems~\ref{thm:cm} and \ref{thm:lam}, we see that a companion matrix can be diagonalized using the Vandermonde matrix of $n$ distinct left eigenvalues (or equiv.\ roots of $p(x)$) when no three lie in the same conjugacy class.

We now come to the question of Jordan normal form of matrices over a division algebra. The following is a consequence of a theorem of Wedderburn (see L.\ Rowen's excellent article \cite{rowen:1990}):
\begin{thm}\label{thm:wedderburn-jordan}
	Let $A\in M_n(\mathcal{D})$ where $\mathcal{D}$ is a central division algebra,  let $F$ be the center of $\mathcal{D}$ and let $f(x)$ be the minimal polynomial of $A$ over the field $F$. Then the following are equivalent:
	\begin{enumerate}
		\item Each irreducible factor of $f$ in $F[\lambda]$ has a root in $\mathcal{D}$;
		\item $f$ factors as a product of linear factors in $\mathcal{D}[\lambda]$;
		\item $A$ is triangularizable over $\mathcal{D}$.
	\end{enumerate}
In case the characteristic of $\mathcal{D}$ in the theorem is $0$, we can add a fourth equivalent statement:
\begin{enumerate}
\setcounter{enumi}{3}
	\item $A$ is conjugate (similar) in $M_n(\mathcal{D})$ to a matrix in Jordan normal form. 
\end{enumerate}
\end{thm}

The last statement in the theorem follows from a more general result in \cite{dokovic:1985} (other sufficient conditions for the existence of the Jordan normal form have  been proved earlier by Cohn~\cite{cohn:1973}).

As can easily be seen, any element $\lambda$ on the diagonal of either the triangular form or Jordan normal form of a matrix in Theorem~\ref{thm:wedderburn-jordan} must be a right eigenvalue of $A$. Moreover, it can be replaced by any element in the conjugacy class of $\lambda$. In particular, for a companion matrix $A$ we can replace $\lambda$, using Theorem~\ref{thm:cm}, by a left eigenvalue in the conjugacy class of $\lambda$, and we choose the same left eigenvalue for all the eigenvalues $\lambda$ in the same conjugacy class. 

We have one final piece required regarding the Jordan normal form of a companion matrix defined over a central division ring. As in the field case, every left eigenvalue $\lambda$ of a companion matrix corresponds to a unique Jordan block. In fact, it is shown in \cite{ChapmanMachen:2017} that any corresponding eigenvector of $\lambda$ is of the form $vc$ where  
\[v=\begin{bmatrix}
	1 \\ \lambda \\ \lambda^2 \\ \vdots \\ \lambda^{n-1}
\end{bmatrix};\]
therefore $A-\lambda{I}$ has a one-dimensional null space, and this is only possible if there is a unique Jordan block for $\lambda$.

\section{Recurrence relations over division algebras} \label{sec:division}

Let $\mathcal{D}$ be a division algebra,
and take a left linear recurrence with coefficients in $\mathcal{D}$. Denote by $C$ and $p(x)$ its companion matrix and primitive characteristic polynomial, respectively. Any root $\lambda$ of $p(x)$ provides a solution $\lambda^k$ to the left recurrence \eqref{eq:linear-recurrence} by Theorem~\ref{thm:cm}. Note however, that unlike the commutative case (see \eqref{eq:superposition}), superposition is valid only when the coefficients appear on the right. That is, given roots $\lambda_1,\ldots,\lambda_t$ of $p(x)$ we get that 
\begin{equation} \label{eq:right-superposition}
	a_k = \lambda_1^k\alpha_1+\cdots+\lambda_t^k\alpha_t
\end{equation}
is a solution for the linear recurrence (without initial conditions) for any $\alpha_1,\ldots,\alpha_t\in \mathcal{D}$. 

\begin{proof}[Proof of Theorem~\ref{thm:diagonal}]
	Under the assumptions of the theorem and using Lam's result (Theorem~\ref{thm:lam}), we get that $C$ is diagonalizable using the Vandermonde matrix, and the proof follows exactly as in the proof of Proposition~\ref{prop:diagonal-field}.
\end{proof}

\begin{exmpl}\label{ex:diagonal}
	Let $a_0 = 1, a_1 = 1, a_{k+2}=(-1-ij)a_{k}+ia_{k+1}$ be a left recurrence with coefficients in $\mathbb{H}$. The corresponding companion matrix is
	\[ A = \begin{pmatrix}
		0 & 1 \\
		-1-ij & i
	\end{pmatrix}
	\]
	and the primitive characteristic polynomial is $p(x) = x^2-ix+ij+1.$ The polynomial has exactly two distinct roots $\lambda=j, \mu=i+j$. 
	
	Define
	\[ V= \begin{bmatrix}
		1 & 1 \\
		j & i+j
	\end{bmatrix}\]
(this is the Vandermonde matrix of the two eigenvalues $\lambda,\mu$). 
This matrix is invertible, and its inverse is
	\[ V^{-1}= \begin{bmatrix}
		1-ij & i \\
		ij & -i
	\end{bmatrix}.\]
We get $A=VDV^{-1}$, where
\[ D= \begin{bmatrix}
		j & 0 \\
		0 & i+j
	\end{bmatrix}.\]
Now we compute the coefficients in \eqref{eq:right-superposition} using $V^{-1}\begin{bmatrix}a_0 \\ a_1 \end{bmatrix} = \begin{bmatrix}1+i-ij\\ ij-1\end{bmatrix}$.
Thus the solution to the recurrence relation is 
\[ a_k = j^k(1+i-ij)+(i+j)^k(ij-1)\] 
\end{exmpl}

\begin{proof}[Proof of Theorem~\ref{thm:jordan}]
	According to the statement of the theorem, the minimal polynomial in $F[x]$ of the matrix $A$ satisfies the first statement of Theorem~\ref{thm:wedderburn-jordan} and therefore there exists an invertible matrix $U$ such that $A=UJU^{-1}$ where $J$ is in Jordan normal form, with $t$ Jordan blocks associated with $t$ distinct left eigenvalues $\lambda_1,\ldots,\lambda_t$ of sizes $m_1,\ldots,m_t$, respectively (in fact the $t$ left eigenvalues are from distinct conjugacy classes). Now the formula for $a_k$ follows through by computing the power $J^k$ using \eqref{eq:kth-power-jordan-block} and then the product
	\[
	UJ^k\begin{bmatrix}
	b_1 \\ b_2 \\ \vdots \\ b_n
	\end{bmatrix}. \]
\end{proof}

\begin{exmpl} \label{ex:jordan1}
Given the recurrence relation
\begin{equation*} 
a_{n+2} = ka_n + (i+j)a_{n+1},	
\end{equation*}
the companion matrix of the recurrence relation is 
\[A=
\begin{pmatrix}
	0 & 1 \\
	k & i+j
\end{pmatrix}
\]
Its primitive characteristic polynomial is $p(x)=x^2-(i+j)x-k=(x-j)(x-i)$. This polynomial has a unique root $x=i$, 
	which is a left eigenvalue of $C$. Indeed, any root of $p(x)$ must be a root of the companion polynomial $C_p(x)=p(x)\cdot\overline{p(x)}=(x^2+1)^2$ 
		(the companion polynomial over the quaternions is actually the characteristic polynomial, as defined in Section~\ref{sec:octonions}, of the companion matrix $C$). Thus, any root must be in the conjugacy class of $i$; but if there was a second root of $p(x)$ conjugate to $i$ it would mean that the root was spherical, implying $j$ was a root, but it is not. For a general treatment on computing roots of quaternionic polynomials see \cite{janovska-opfer:2010}.
	
We obtain the Jordan form of $A$. An eigenvector of $A$ associated with $i$ is $v = \begin{bmatrix}
	1 \\ i
\end{bmatrix}$. 
A generalized eigenvector $w$ solving $Aw-wi = v$ is $w = \begin{bmatrix}
	-\frac12j \\ \frac12{k}+1
\end{bmatrix}$.

One can check that in fact if we take 
\[ U = 
\begin{pmatrix}
1 & -\frac{j}{2} \\
i & 	1+\frac{k}{2}
\end{pmatrix}, 
\quad U^{-1} = 
\begin{pmatrix}
\frac{3}{4}+\frac{k}{4} & -\frac{i}{4}+\frac{j}{4} \\
-\frac{i}{2}+\frac{j}{2} & \frac{1}{2}-\frac{k}{2}
\end{pmatrix},
\]
we get 
\[U^{-1}AU = \begin{pmatrix}
i & 1 \\
0 & i	
\end{pmatrix}.\]
We compute the $n$th power of $A$:

\begin{align} \label{eq:ex-jordan}
A^n & = U\begin{pmatrix}
i & 1 \\
0 & i	
\end{pmatrix}^nU^{-1}
 = U\begin{pmatrix}
i^n & ni^{n-1} \\
0 & i^n	
\end{pmatrix}U^{-1}.
\end{align}
\end{exmpl}

We set $\begin{bmatrix}
	b_0 \\ b_1
\end{bmatrix} = U^{-1} \begin{bmatrix}
	a_0 \\ a_1
\end{bmatrix}$ (where $a_0,a_1$ are the initial conditions of the linear recurrence). Multiplying through the matrices in \eqref{eq:ex-jordan} we get 
\[a_{n} = i^nb_0 + (-ni-\frac{j}{2})i^nb_1.\]

\section{Recurrence relations over the Octonions} \label{sec:octonions}

In this section, we show how to solve homogeneous left linear recurrences of order $2$ with coefficients in division octonion algebras. We start by recalling the construction of octonion algebras using the Cayley--Dickson doubling process.

Let $F$ be a field and $A$ be an $F$-algebra with involution $\sigma$.
The Cayley--Dickson doubling process requires a choice of an element $\gamma \in F^\times$ and produces an algebra $B=A \oplus A\ell$ with multiplication defined by
$$(q+r\ell)(s+t\ell)=qs+\gamma \sigma(t)r+(tq+r\sigma(s))\ell, \quad \forall q,r,s,t \in A.$$
We denote the algebra $B$ by the pair $(A,\gamma)$. 
The involution $\sigma$ extends to $B$ by $\sigma(q+r\ell)=\sigma(q)-r\ell$, and satisfies the identity $\sigma(xy)=\sigma(y)\sigma(x)$. For example, $(\mathbb{C},-1) = \mathbb{H}$, Hamilton's quaternions, and  $(\mathbb{\mathbb{H}},-1) = \mathbb{O}$, Cayley's octonions. 
(See \cite{Schafer:1954} for a general treatise on Cayley--Dickson algebras). 
More generally, an $8$ dimensional $F$-algebra generated by the Cayley--Dickson process is called an octonion algebra. 

Octonion algebras are non-associative, so that we cannot use the tools from \S\ref{sec:division-prelim}. But in what follows we prove that for order~$2$ left linear recurrences over an octonion division algebra, we can reduce the problem to solving matrix relations over quaternion algebras. The main fact in our favor is that octonion algebras are alternative, i.e.\ any two elements generate a subalgebra of quaternions which is associative (see \cite[p.\ 47]{schafer1966introduction}).

Denote by $O$ an octonion division algebra. We denote its involution $\sigma(\tau)$ by $\bar{\tau}$ for short. Let $p(x)\in O[x]$ be a polynomial, defined in the same way we defined them for division algebras, with $x$ central and coefficients on the left. The (octonion) conjugate polynomial $\overline{p(x)}$ of $p(x)$ is defined by conjugating all the coefficients using the involution $\sigma$. 
We define the companion polynomial $C_p(x)=p(x)\overline{p(x)}$. This polynomial has real coefficients, and is in fact the characteristic polynomial of the companion matrix associated with the polynomial $p(x)$ as defined in Section~\ref{sec:division-prelim} (see \cite{chapman2020}, \cite{chapman2020b} on octonion polynomials, and \cite{gauss} for polynomials over general Cayley--Dickson algebras). 

 Let $a_0 = \tau, a_1 = \nu$ and $a_{n+2}=\alpha{a_n}+\beta{a_{n+1}}$ where $\tau,\nu,\alpha,\beta \in {O}$. The companion matrix of this order~$2$ left linear recurrence is 
\[ A = \begin{pmatrix}
0 & 1 \\
\alpha & \beta 
\end{pmatrix},
\]
with primitive characteristic polynomial $p(x)=x^2-\beta{x}-\alpha$.

As in the classical case,  
\[
\begin{bmatrix}
	a_{n+1} \\
	a_{n+2}
\end{bmatrix} = 
A\begin{bmatrix}
	a_{n} \\
	a_{n+1}
\end{bmatrix}.
\]
Iterating backwards, we get
\[
\begin{bmatrix}
	a_{n} \\
	a_{n+1}
\end{bmatrix} = 
\underbrace{A(A(\cdots (A}_{n\text{ times}}
\begin{bmatrix}
	\tau \\
	\nu
\end{bmatrix})\cdots)).
\]

The coefficients $\alpha,\beta$ generate an associative subalgebra of $O$, as $O$ is alternative, contained in a quaternion subalgebra we denote by $Q$. Write ${O} = Q\oplus Q\ell$ for some element $\ell\in{O}$ of trace $0$. Thus, we can write 
\[ \tau = q+s\ell, \nu = r+t\ell,\]
with $q,s,r,t\in Q$. 

\begin{lem}
Using the notation $Q$ and $\ell$ defined above, 
let $B=\begin{pmatrix}
x & y \\
z & w 
\end{pmatrix}\in M_2(Q)$ and let $v=\begin{bmatrix}
	m \\
	n
\end{bmatrix} \in Q^2.$
Then $B(v\cdot\ell) = \overline{(\bar{B}\bar{v})}\ell$ (where $\bar\square$ denotes quaternion conjugation). \end{lem}

\begin{proof} This is a straightforward computation:
	\[
	\begin{aligned}
		B(v\cdot\ell) &= \begin{pmatrix}
			x & y \\
			z & w 
		\end{pmatrix}\begin{bmatrix}
			m\ell \\
			n\ell
		\end{bmatrix} = \begin{bmatrix}
			x(m\ell)+y(n\ell) \\
			z(m\ell)+w(n\ell)
		\end{bmatrix} = \\
		 &= \overline{\begin{bmatrix}
				\bar{x}\bar{m}+\bar{y}\bar{n} \\
				\bar{z}\bar{m}+\bar{w}\bar{n}
		\end{bmatrix}}\ell = \overline{(\bar{B}\bar{v})}\ell.
	\end{aligned}
	\]
\end{proof}

Now, 
\[ A\begin{bmatrix}
	\tau \\
	\nu
\end{bmatrix} = A\begin{bmatrix}
	q \\
	r
\end{bmatrix} + A\begin{bmatrix}
	s\ell \\
	t\ell
\end{bmatrix}\]
Using the Lemma, we see that 
\[ A\begin{bmatrix}
	\tau \\
	\nu
\end{bmatrix} = A\begin{bmatrix}
	q \\
	r
\end{bmatrix} + \overline{\left(\bar{A}\begin{bmatrix}
	\bar{s} \\
	\bar{t}
\end{bmatrix}\right)} \ell\]

Iterating, we get by induction that

\[A \left( A \left( \cdots A\begin{bmatrix}
	\tau \\
	\nu
\end{bmatrix} \right) \cdots \right) = A^n\begin{bmatrix}
	q \\
	r
\end{bmatrix} + \overline{\left((\bar{A})^n\begin{bmatrix}
	\bar{s} \\
	\bar{t}
\end{bmatrix}\right)} \ell\]

Both $A$ and $\bar{A}$ are companion matrices over $Q$ (this also makes the expression $A^n$ unambiguous). It therefore remains to find their normal forms, as in \S\ref{sec:division}.  A direct computation then gives us the required formula for $a_k$.
\begin{rem}
	Note that, in general, for two matrices $A,B$ with coefficients in a quaternion algebra $Q$ (and other noncommutative rings) we have $\overline{AB} \ne \bar{A}\bar{B}$ and $\overline{AB} \ne \bar{B}\bar{A}$.
\end{rem}

\begin{proof}[Proof of Theorem~\ref{thm:octonion}]
As in the statement of the theorem, let $\lambda_1,\mu_1$ be the distinct roots of the primitive characteristic polynomial $p(x)$. If $\lambda_1,\mu_1$ are in the same conjugacy class, then in fact the roots are spherical, and all elements in the conjugacy class are roots (see \cite{chapman2020, chapman2020b}). Moreover, this means $p(x)$ and the linear recurrence are defined over $F$, and the problem reduces to solving the recurrence over the subalgebra of quaternions generated by $\tau$ and $\nu$. We will therefore assume now that $\lambda_1,\mu_1$ belong to two distinct conjugacy classes.  

Let $\bar{A}$ be the conjugate matrix of $A$ (elementwise octonion conjugation), i.e.
\[ \bar{A} = \begin{pmatrix}
0 & 1 \\
\bar{\alpha} & \bar{\beta} 
\end{pmatrix}.
\]
Clearly $\bar{A}$ is a companion matrix, and its primitive characteristic polynomial is $\overline{p(x)}=x^2-\bar{\beta}{x}-\bar{\alpha}$. Let $C_p(x)=p(x)\overline{p(x)}\in\mathbb{R}[x]$ be the companion polynomial of $p(x)$; it has the property that any root of $p(x)$ is also a root of $C_p(x)$ and any root of $C_p(x)$ is conjugate to some root of $p(x)$ (see \cite{ChapmanMachen:2017}). Therefore $\overline{p(x)}$ has two roots from distinct conjugacy classes just like $p(x)$, which we denote by $\lambda_2,\mu_2$. Now $A$ and $\bar{A}$ are diagonalizable using the Vandermonde matrices $V_i(\lambda_i,\mu_i), i=1,2$, respectively, i.e.,
\begin{equation} \label{eq:octonion-decomp}
	A^n\begin{bmatrix}
	\tau \\
	\nu
\end{bmatrix} = V_1D_1^nV_1^{-1}\begin{bmatrix}
	q \\
	r
\end{bmatrix} + \overline{\left(V_2D_2^nV_2^{-1}\begin{bmatrix}
	\bar{s} \\
	\bar{t}
\end{bmatrix}\right)} \ell,
\end{equation}
where
\[
	D_1 = \begin{bmatrix}
		\lambda_1 & 0 \\
		0 & \mu_1
	\end{bmatrix}, \quad 
	D_2 = \begin{bmatrix}
		\lambda_2 & 0 \\
		0 & \mu_2
	\end{bmatrix}.
\]

Setting
	\begin{equation}
		\begin{bmatrix}
			b_1 \\ b_2 
		\end{bmatrix} = V_1^{-1}	
		\begin{bmatrix}
			q \\ r
		\end{bmatrix}, \quad
		\begin{bmatrix}
			c_1 \\ c_2 
		\end{bmatrix} = V_2^{-1}	
		\begin{bmatrix}
			s \\ t
		\end{bmatrix},
\end{equation}
 we get
 \begin{align*}
 a_n &= \lambda_1^nb_1+\mu_1^nb_2+\overline{\lambda_2^nc_1+\mu_2^nc_2}\ell \\
	 &= \lambda_1^nb_1 + \mu_1^nb_2 + (\bar{c}_1\bar{\lambda}_2^n+\bar{c}_2\bar{\mu}_2^n)\ell 	
 \end{align*}
 
\end{proof}

\begin{exmpl}
	We take a very simple example of a left linear recurrence of order $2$ defined over $\mathbb{O}$: $$a_{k+2}=(-1-ij)a_{k}+ia_{k+1}$$ with initial conditions $a_0 = 1, a_1 = \ell$. We have already seen the companion matrix in Example~\ref{ex:diagonal}:  	
	\[ A = \begin{pmatrix}
		0 & 1 \\
		-1-ij & i
	\end{pmatrix}
	\]
The expression $A \left( A \left( \cdots A\begin{bmatrix}
	1 \\
	\ell
\end{bmatrix} \right) \cdots \right)$ becomes 
\[A^n\begin{bmatrix}
	1 \\
	0
\end{bmatrix} + \overline{\left((\bar{A})^n\begin{bmatrix}
	0 \\
	1
\end{bmatrix}\right)} \ell.\]

We remember from Example~\ref{ex:diagonal} that $A^n = V_1D_1^nV_1^{-1}$ where 
	\[ V_1= \begin{bmatrix}
		1 & 1 \\
		j & i+j
	\end{bmatrix}, \quad 
	D_1= \begin{bmatrix}
		j & 0 \\
		0 & i+j
	\end{bmatrix}, \quad
	V_1^{-1}= \begin{bmatrix}
		1-ij & i \\
		ij & -i
	\end{bmatrix}.\]
Since we only care about the first entry of the resulting vector, we see that the contribution of $A^n\begin{bmatrix}
	1 \\
	0
\end{bmatrix}$ to this entry is $j^n(1-ij) + (i+j)^nij$. The primitive characteristic polynomial of $\bar{A}$ is $\overline{p(x)} = x^2+ix+(1-ij)$. This polynomial has exactly two roots $j$ and $j-i$. So we get $\bar{A}^n=V_2D_2^nV_2^{-1}$ where
 	\[ V_2= \begin{bmatrix}
		1 & 1 \\
		j & j-i
	\end{bmatrix}, \quad 
	D_2= \begin{bmatrix}
		j & 0 \\
		0 & j-i
	\end{bmatrix}, \quad
	V_2^{-1}= \begin{bmatrix}
		1+ij & -i \\
		-ij & i
	\end{bmatrix}.\]
Computing the first entry of $\overline{\left((\bar{A})^n\begin{bmatrix}
	0 \\
	1
\end{bmatrix}\right)} \ell,$ we get 
\[\overline{(-j^ni+(j-i)^ni}\ell = (i(-j)^n - i(i-j)^n)\ell.\]
So the solution to the left linear recurrence is 
\[a_n = j^n(1-ij) + (i+j)^nij + (i(-j)^n - i(i-j)^n)\ell.\]
\end{exmpl}

\begin{proof}[Proof of Theorem~\ref{thm:octonion2}]
	The proof follows the same idea as in the proof of Theorem~\ref{thm:octonion}. One shows that if $p(x)$ has a unique root in $O$ then $\overline{p(x)}$ also has a unique root in $O$. Both quadratic polynomials split over $O$ guaranteeing the existence of Jordan forms for $A$ and $\bar{A}$,  
		and the rest proceeds in a straightforward way from the decomposition in \eqref{eq:octonion-decomp}.
\end{proof}

\begin{exmpl}
	We solve the recurrence relation
\begin{equation*} 
a_{n+2} = ka_n + (i+j)a_{n+1},	
\end{equation*}
with initial conditions $a_0 = 1, a_1 = \ell$. As shown in Example \ref{ex:jordan1}, the companion matrix \(A=
\begin{pmatrix}
	0 & 1 \\
	k & i+j
\end{pmatrix}
\) has Jordan form \(J_1=\begin{pmatrix}
i & 1 \\
0 & i	
\end{pmatrix}\), with transition matrix \(U_1 = 
\begin{pmatrix}
1 & -\frac{j}{2} \\
i & 	1+\frac{k}{2}
\end{pmatrix}\). The conjugate matrix \(\bar{A} = \begin{pmatrix}
	0 & 1 \\
	-k & -i-j
\end{pmatrix}\) has Jordan form \(J_1=\begin{pmatrix}
-j & 1 \\
0 & -j	
\end{pmatrix}\) with transition matrix \(U_2 = \begin{pmatrix}
1 &     \frac{i}{2} \\
-j & 	1-\frac{k}{2}
\end{pmatrix}\). We compute
\[U_1^{-1}\begin{bmatrix}
	1 \\
	0
\end{bmatrix} = \begin{bmatrix}
	\frac34+\frac{k}{4} \\
	-\frac{i}{2}+\frac{j}{2}
\end{bmatrix},\quad U_2^{-1}\begin{bmatrix}
	0 \\
	1
\end{bmatrix} = \begin{bmatrix}
	-\frac{i}{4}+\frac{j}{4} \\
	\frac12+\frac{k}{2}
\end{bmatrix}.\]
Using the same decomposition as in \eqref{eq:octonion-decomp} we get
\tiny
\[
\begin{aligned} 
a_n &= i^n(\frac34+\frac{k}4) + (-ni-\frac{j}2)i^n(-\frac{i}2+\frac{k}2) +
\overline{(-j)^n(-\frac{i}4+\frac{j}4) + (nj+\frac{i}2)(-j)^n(\frac12+\frac{k}2)}\ell \\
&= i^n(\frac34+\frac{k}4) + (-ni-\frac{j}2)i^n(-\frac{i}2+\frac{k}2) + 
((\frac{i}4-\frac{j}4)j^n+(\frac12-\frac{k}2)j^n(-nj-\frac{i}2))\ell
\end{aligned} 
\]
\normalsize
\end{exmpl}

\bibliographystyle{abbrv}
\bibliography{bibfile}

\vspace*{2em}

\end{document}